\documentclass[a4paper,12pt]{article}
\usepackage{amsmath,amssymb,enumerate,amsthm}
\usepackage{color}

\usepackage{hyperref}
\hypersetup{
setpagesize=false,
 bookmarksnumbered=true,%
 bookmarksopen=true,%
 colorlinks=true,%
 linkcolor=blue,
 citecolor=red
}

\newcommand{\R}{\mathbb{R}}

\newcommand{\N}{\mathbb{N}}

\newtheorem{theorem}{Theorem}[section]
\newtheorem{lemma}[theorem]{Lemma}
\newtheorem{proposition}[theorem]{Proposition}
\newtheorem{corollary}[theorem]{Corollary}
\theoremstyle{remark}
\newtheorem{remark}{Remark}[section]
\theoremstyle{definition}

\newtheorem{definition}{Definition}[section]

\numberwithin{equation}{section}

\makeatletter
\def\@cite#1#2{[{{\bfseries #1}\if@tempswa , #2\fi}]}
\makeatother                  

\begin{document}
\begin{center}
\Large{{\bf 
Domain characterization 
for Schr\"odinger operators with 
sub-quadratic singularity
}}
\end{center}

\vspace{5pt}

\begin{center}
Giorgio Metafune%
\footnote{
Dipartimento di Matematica e Fisica “Ennio De Giorgi”, Universit\`a del Salento, C.P.193, 73100, Lecce, Italy, 
E-mail:\ {\tt giorgio.metafune@unisalento.it}}
and 
Motohiro Sobajima%
\footnote{
Department of Mathematics, 
Faculty of Science and Technology, Tokyo University of Science,  
2641 Yamazaki, Noda-shi, Chiba, 278-8510, Japan,  
E-mail:\ {\tt msobajima1984@gmail.com}}
\end{center}

\newenvironment{summary}{\vspace{.5\baselineskip}\begin{list}{}{%
     \setlength{\baselineskip}{0.85\baselineskip}
     \setlength{\topsep}{0pt}
     \setlength{\leftmargin}{12mm}
     \setlength{\rightmargin}{12mm}
     \setlength{\listparindent}{0mm}
     \setlength{\itemindent}{\listparindent}
     \setlength{\parsep}{0pt}
     \item\relax}}{\end{list}\vspace{.5\baselineskip}}
\begin{summary}
{\footnotesize {\bf Abstract.}
We characterize the domain of  the Schr\"odinger operators $S=-\Delta+c|x|^{-\alpha}$ in $L^p(\R^N)$, 
with $0<\alpha<2$  and $c\in\R$. 
When $\alpha p< N$, the domain characterization is essentially known and
can be proved using different tools, for instance kernel estimates and potentials in the 
Kato class or in  the reverse H\"older class. 
However, the other cases seem not to be known, so far.
In this paper, we give the explicit description of the domain 
of  $S$ for all range of parameters $p,\alpha$ and $c$. 
}
\end{summary}

{\footnotesize{\it Mathematics Subject Classification}\/ (2020): 35J75, 35K67, 47D08.

}



\section{Introduction} 
In this paper we consider Schr\"odinger operators 
with singular potentials of the form
\[
 S=-\Delta +\frac{c}{|x|^{\alpha}} \quad\text{in}\ \R^N 
\]
in the Lebesgue space $L^p=L^p(\R^N)$
where $N\geq 2$, $0<\alpha<2$, $c\in\R$ and $1<p<\infty$.
In particular our analysis applies to the Coulomb potential 
corresponding to $\alpha=1$, both in the attracting and repulsive case, depending on the sign of $c$.
Here we focus our attention 
on {\it the characterization of the domain} of 
the reasonable realization $S_p$ of $S$ in $L^p$.
We do not consider the case $N=1$, which should be treated on the half line $[0, \infty[$ and is slightly different because of boundary conditions at $x=0$.

From the viewpoint of the scale homogeneity, the potential $c|x|^{-2}$  
has the same homogeneity of the Laplacian. 
Therefore 
 the case $\alpha=2$ is expected to be critical in some sense.  
Actually, the situations for three cases $0<\alpha<2$, $\alpha=2$
and $\alpha>2$ are completely different from each other. 

In the case $p=2, N=3$, Schr\"odinger operators with  
Coulomb potentials  have been  studied by Kato in \cite{Kato}  who proved, via the Kato-Rellich perturbation theorem, 
that the operator $S_2$ endowed with domain $H^2(\R^3)$ is selfadjoint.
This perturbation theory has been extended to the $L^p$ setting 
by  Okazawa \cite{Okazawa1996} and Davies--Hinz \cite{DaviesHinz1998}). 
Actually, as a consequence of the Rellich inequality, 
on can prove that if $1<p<\frac{N}{2}$, then 
$S_p$ endowed with domain $W^{2,p}(\R^N)$ 
is quasi-$m$-accretive in $L^p$ if $0<\alpha <2$ and $c \in \R$. 
In Section 6 we employ such an approach to reach $1<p< \frac N \alpha$.

If $\alpha >2$ and $c>0$, Okazawa proved in \cite[Theorem 2.1]{Okazawa1984} that 
the domain of $S_p$ in $L^p$ $(1<p<\infty)$,  is the intersection 
between the domain of the Laplacian and that of the potential, namely, 
\begin{equation} \label{domainreg}
D(S_p)=W^{2,p}(\R^N) \cap D(|x|^{-\alpha})=\{u \in W^{2,p}(\R^N): |x|^{-\alpha} u \in L^p\}.
\end{equation}

Since the case $\alpha=2$ is critical, 
one needs $c \geq -(N-2)^2/4$ in order that $S$ is semibounded. 
In this case $-S_p$ generates a semigroup only for certain values of $p$ and the domain characterization depends on $p$, \cite[Sections 3,4]{Okazawa1996} or \cite[Examples 7.1, 7.2]{MOSS-inverse2}.
Moreover, in some case, the realization of $S$ as a (negative) generator 
of positive $C_0$-semigroup in $L^p$ is not unique (see e.g., \cite{MSSnonunique}). 

Now we go back to the case $\alpha<2$ where we have not been able to find general results even in the $L^2$ setting. 

One can prove that  \eqref{domainreg} holds if and only if $\alpha p <N$ 
(and in this case $D(S_p)=W^{2,p}(\R^N)$) by the following argument. 
Since the potential $V(x)=c|x|^{-\alpha}$ is in the Kato class, 
see \cite[Proposition A.2.5]{Simon1982}, 
the associated semigroup $e^{-tS}$ satisfies upper and lower Gaussian estimates, 
\cite[Section B.7]{Simon1982} so that, 
taking the Laplace transform of the semigroup, the integral kernel of $(\lambda +S)^{-1}$ 
is comparable to that of $(\lambda-\Delta)^{-1}$. 
Therefore $V(\lambda+S)^{-1}$ is bounded in $L^p$ if and only if 
the same holds for $V(\lambda-\Delta)^{-1}$. 
This means that the multiplication by $|x|^{-\alpha}$ is bounded from $W^{2,p}(\R^N)$ to $L^p$ 
and requires $\alpha p <N$. 
Methods as in \cite{Shen1995}, which essentially prove \eqref{domainreg} 
under reverse H\"older conditions on the potential, lead to the same restriction.

The purpose of this paper is to characterize the domain of a suitable realization $S_p$ of 
$S$ for all $0<\alpha<2$, 
$c\in \R$ and $1<p<\infty$. 
To clarify the strategy, we give a proof also for the known cases. 
As a consequence, 
it turns out that for the case $p \geq N/\alpha$, 
the domain of $S_p$ differs from the usual one $W^{2,p}(\R^N)$
but only for  a finite dimensional space which depends on 
the power of singularity $\alpha$ and also the constant $c$ in front of $|x|^{-\alpha}$. 
More precisely, 
the description of the domain of $S_p$ employs the functions
\begin{gather*}
\phi(x)=\sum_{k=0}^{m}
\frac{c^k\Gamma(\frac{N-\alpha}{2-\alpha})}{(2-\alpha)^{2k} k!\Gamma(\frac{N-\alpha}{2-\alpha}+k)}
|x|^{(2-\alpha)k}.
\\
\phi_j(x)
=
x_j
\sum_{k=0}^{m}
\frac{c^k\Gamma(\frac{N}{2-\alpha}+1)}{(2-\alpha)^{2k}k!\Gamma(\frac{N}{2-\alpha}+1+k)}
|x|^{(2-\alpha)k}.
\end{gather*}
We use the first when $\frac N p \leq \alpha < \frac N p+1$ and both when $\alpha \geq \frac N p+1$ to capture the singularity near the origin of function in the domain. In fact, $\phi \in W^{2,p}_{loc}(\R^N)$ if and only if $\alpha < \frac N p$ and $\phi_j \in W^{2,p}_{loc}(\R^N)$ if and only if $\alpha < \frac N p+1$ . However, $S\phi$ and $S\phi_j$ are bounded near the origin if $m$ is sufficiently large.

{\bf Notation.} We use $L^p, W^{k,p}$ for $L^p(\R^N), W^{k,p}(\R^N)$. $B_r=B(0,r)$ is the ball in $\R^N$ centred at 0 and with radius $r>0$ and we write $B$ for $B_1$.
We write $A_p$ for the $L^p$ realization of a differential operator $A$.
\section{Preliminary results}
\subsection{The operator $S$  in spherical coordinates} \label{preliminaries}
We introduce spherical coordinates
\begin{equation*}
\left\{
\begin{array}{ll}
  x_1=r \cos\theta_1\sin \theta_2\ldots\sin \theta_{N-1}\\ 
  x_2=r \sin\theta_1\sin \theta_2\ldots\sin \theta_{N-1}\\
 \vdots\\
 x_n=r\cos \theta_{N-1}\end{array}\right.
\end{equation*}
where $\theta_2,\ldots,\theta_{N-1}$ range from $0$ to $\pi$ and $\theta_1$ ranges from $0$ to $2\pi$.
The Laplace operator is then  given by 
$$\Delta= \frac{\partial^2}{\partial r^2}+\frac{N-1}{r}\frac{\partial }{\partial r}+\frac{1}{r^2}\Delta_0$$
where 
$$\Delta_0=\frac{1}{\sin^{N-2}\theta_{N-1}}\frac{\partial}{\partial\theta_{N-1}}\sin^{N-2}\theta_{N-1}\frac{\partial}{\partial\theta_{N-1}}+\ldots+\frac{1}{\sin^2\theta_{N-1}\cdots\sin^2\theta_2}\frac{\partial^2}{\partial\theta_1^2}$$
 is the Laplace-Beltrami operator on the unit sphere $S^{N-1}$, see \cite[Chapter IX]{vilenkin}) . 

We recall that a spherical harmonic $P_n$ of order $n$ is the restriction to $S^{N-1}$ of a homogeneous harmonic polynomial of degree $n$ and that the linear span of spherical harmonics (which coincides with all polynomials) is  dense in $C(S^{N-1})$, hence in $L^p(S^{N-1})$.

\begin{lemma} 
 Let $P_n$ be a spherical harmonic of degree $n$. Then 
$$\Delta_0 P_n=-(n^2+(N-2)n)P_n.$$ 
The values $\lambda_n:=n^2+(N-2)n$ are the eigenvalues of the Laplace-Beltrami operator $-\Delta_0$ on $S^{N-1}$. 
The corresponding eigenspace consists of all spherical harmonics of degree $n$ and has dimension $d_n$ where $d_0=1, d_1=N$ and
$$
d_n= \binom{N+n-1}{n}-\binom{N+n-3}{n-2}
$$
for $n \ge 2$.
\end{lemma}
If $u\in C_c^\infty(\R^N \setminus \{0\})$, $u(x)=\sum c_n(r)P_n(\omega)$ (here we consider finite sums), then
\begin{align} \label{represent}
&-Su(r,\omega)=\sum \left(c''_n(r)+\frac{N-1}{r}c'_n(r)-\left (\frac {c}{r^\alpha}+\frac{\lambda_n}{r^2}\right )c_n(r)\right)P_n(\omega)
\end{align}
where the eigenvalues $\lambda_n$ are repeated according to their multiplicity.

\subsection{The spaces $L^p_{J}$}
If $X,Y$ are function spaces over $G_1, G_2$ we denote by $X\otimes Y$ the algebraic tensor product of $X,Y$, that is the set of all functions $u(x,y)=\sum_{i=1}^n f_i(x)g_i(y)$ where $f_i \in X, g_i \in Y$ and $x \in G_1, y\in G_2$.
In what follows  we denote by $P$ a spherical harmonic and by ${\rm deg}\ P$ its degree. We fix a complete orthonormal  system of spherical harmonics $\{P_j, j \in \N_0\}$ (which is dense in $ L^p(S^{N-1})$ for every $1 \le p < \infty$) and a subset $J$ of $\N_0$.

 When $J \subset \N_0$
$L^p_{J}$ $(1\leq p<\infty)$ is the closure of 
$$L^p(]0,+\infty[, r^{N-1}d\rho)\otimes span\{P_j: j \in J\}$$  in $L^p(\R^N)$.  We use $L^p_{\ge n}$, $L^p_{< n}$ and $L^p_{n}$ when $J$ identifies all spherical harmonics of degree $ \ge n$, $<n$, $=n$, respectively.

Let us observe that $L^p=L^p_{\N_0}$ and that $L^p_0$ consists of all radial functions in $L^p$. Moreover
$C_c^\infty (]0,\infty[)\otimes  span\{P_j: j \in J\}$ is dense in $L^p_J$ for $1 \le p < \infty$.
Observe that, by (\ref{represent}), the spaces $L^p_J$ are invariant under the operator $L$. 

The next results clarify the structure of the spaces  $L^p_{J}$. We refer to \cite[Section 2]{MSS-Rellich-disconti} for all proofs.

\begin{lemma} \label{projection} Let $1 \le p \le \infty$ and assume that the $L^2$ orthogonal projection $S:L^2(S^{N-1}) \to span\{P_j: j \in J\}$ extends to a bounded projection in $L^p(S^{N-1})$. Then
$$
L^p=L^p_{J} \oplus L^p_{\N_0 \setminus J}
$$
and
\begin{equation}   \label{inclusione}
 L^p_{J}=\left \{u \in L^p: \int_{S^{N-1}} u(r\, \omega)P_j(\omega) \, d\sigma (\omega)=0\ {\rm for} \ r>0\ {\rm  and}\  j\not \in J \right \}.
\end{equation}
If $J$ is finite we have in addition  
$$
L^p_{J}=\Bigl \{ u=\sum_{ j \in J}f_j(r)P_j(\omega): f_j \in L^p(]0,+\infty[, r^{N-1}d\rho)\Bigr \}
$$ and 
the projection $I\otimes S :L^p \to L^p_{ J}$ is given by
$$
(I\otimes S) u=\sum_{j \in  J}T_j u (r)\, P_j(\omega),
$$
where $$T_j u (r):=\int_{S^{N-1}} u(r\, \omega)P_j(\omega) \, d\sigma (\omega).$$
\end{lemma}

\begin{remark}{\rm 

Observe that the hypotheses on the above lemma are always satisfied if $p=2$ or if $J$ (or $\N_0\setminus J$) is finite. In this last case note also that if $u \in C_c^\infty (\R^N \setminus \{0\})$  then $(I\otimes S)u \in C_c^\infty (\R^N \setminus \{0\})$.
We also remark that equality (\ref{inclusione}) holds without assuming the existence of a bounded projection.
}
\end{remark}

We denote by $W^{k,p}_0$ the closure of $C_c^\infty(\R^N \setminus \{0\})$ in $W^{k,p}$.

\begin{lemma}\label{decomposition}
Assume that $N \geq 2$.  Let  $w_0\in C_c^\infty(\R^N)$ be a radial function satisfying 
$w_0=1$ on $B(0,1)$ and $w_0=0$ on $B(0,2)^c$ and  
\[
w_j(x)=x_jw_0(x)=r \omega_j w_0(r), \quad j=1,\ldots, N.
\]
Then 
\begin{itemize}
\item[(i)] If $ 1 \leq p \leq N$, then $W^{1,p}=W^{1,p}_0$;
\item[(ii)] If $N < p<\infty$, then $W^{1,p}=W^{1,p}_0\oplus span\{w_0\}$;
\item[(iii)]
If $1 \leq p\leq N/2$, then $W^{2,p}=W_0^{2,p}$; 
\item[(iv)]
if $N/2<p\leq N$, then 
$W^{2,p}=W^{2,p}_0\oplus span\{w_0\}$;
\item[(v)]
if $N<p<\infty$, then 
$W^{2,p}=W^{2,p}_0 \oplus span\{w_0,w_1,\cdots,w_N\}$. 
\end{itemize}
\end{lemma}
\begin{proof} We give a proof for $W^{2,p}$, that for $W^{1,p}$ is similar and easier.

Let $u \in C_c^\infty (\R^N)$ and define $v=u$ if $1\le p  \le N/2$,\  $v=u-u(0)w_0$ if $N/2 < p\le N$\  and $v=u-u(0)w_0 -\sum_i u_{x_i}(0)w_i$ if $p>N$. We have to prove that $v \in W^{2,p}_0$.

Fix $\eta\in C_c^\infty(\R^N;[0,1])$ such that
$\eta=0$ in $B(0,1)$ and $\eta=1$ in $\R^N\setminus B(0,2)$. 
Define $v_k(x):=\eta(kx)v(x)$, $k\in\N$. Clearly, we have $v_k\in C_c^\infty(\R^N\setminus\{0\})$ 
and $v_k\to v$ in $L^p$.
Using $|v(x)| \le C|x|$, $|x| \le 1$, for $N/2 <p \leq N$ and $|v(x)| \le C|x|^2$, $|\nabla v(x)| \le C|x|$, $|x| \le 1$, for $p>N$, one shows that $v_k \to v$ in $W^{2,p}$   (weakly for $p=N/2,N$) and concludes the proof.
\end{proof}

The above lemma is used to show that $I\otimes S$ is bounded in $W^{2,p}$. Since the proof employs 
the Calder\'on-Zygmund inequality, we exclude $p=1$.

\begin{lemma} \label{ItensorP}
Let $1<p<\infty$ and assume that $J$ is finite. Then the  projection $I\otimes S$  of Lemma \ref{projection} extends to a bounded projection of $W^{2,p}$. 
\end{lemma}

The heat semigroup $e^{t \Delta}$ preserves the spaces $L^p_J$, for every $J \subset \N_0$.
\begin{lemma} \label{heatLpn}
Let $1 \le p \le \infty$, $J \subset \N_0$. Then  $e^{t \Delta }L^p_J \subset L^p_J$. If $J$ is finite, the  projection $I\otimes S$  of Lemma \ref{projection} satisfies for every $u \in L^p$
\begin{equation} \label{commute}
\Delta (I \otimes S) u=(I\otimes S)\Delta u, \  u \in W^{2,p} \qquad e^{t\Delta} (I\otimes S)u=(I\otimes S)e^{t\Delta} u,\ u \in L^p.
\end{equation}
\end{lemma}

Next we define the spaces 
\[
W^{m,p}_{J}=W^{m,p} \cap L^p_{J}, \quad W^{m,p}_{\geq n}=W^{m,p} \cap L^p_{\geq n}, 
\]
 and we state the following density result. 

\begin{lemma} \label{density}
Let $1 \le p <  \infty$. Then $C_c^\infty(\R^N)$ functions of the form
\begin{equation}  \label{density-support}
v=\sum f_j(r) P_j(\omega),
\end{equation}
where the sums are finite and  $j \in J$, are dense in $W^{m,p}_J$ with respect to the Sobolev norm. If $u \in C_c^\infty (\R^N \setminus\{0\})$, the approximating functions of the form (\ref{density-support}) can be chosen to have support in $\R^N \setminus\{0\}$, too.
\end{lemma}

\begin{proposition} \label{density1}
Let $1 \le p < \infty$. Then $C^\infty_c(\R^N \setminus \{0\}) \cap L^p_{\geq 1}$ is dense in $W^{1,p}_{\geq 1}$ and $C^\infty_c(\R^N \setminus \{0\}) \cap L^p_{\geq 2}$ is dense in  $W^{2,p}_{\geq 2}$. In particular $W^{1,p}_{ \geq 1} \subset W^{1,p}_0$ and $W^{2,p}_{ \geq 2} \subset W^{2,p}_0$ .
\end{proposition}
This is proved in \cite[Corollary 2.10]{MSS-disconti-gen} for $W^{1,p}_{ \geq 1}$  and in \cite[Proposition 3.6]{MSS-disconti-gen} for $W^{2,p}_{\geq 2}$.

\smallskip

Finally, we quote a well-known result (see \cite[Lemma 2.11]{MSS-disconti-gen} for a proof). It says that the integration by parts formula holds for  $\phi \in C_c^\infty(\R^N)$ whenever it holds for every $\phi \in C_c^\infty (\R^N \setminus \{0\})$, $N \geq 2$. This  is  false if $N=1$.
\begin{lemma} \label{r2}
If $1 \le p < \infty$ and $N \ge 2$, then $W^{k,p}(\R^N \setminus\{0\})=W^{k,p}(\R^N)$.
\end{lemma}

\section{Hardy inequalities}
We recall the  classical Hardy inequality
\begin{theorem}\label{lem:Hardy}
If $1 \leq p<N$, then for every $u\in W^{1,p}$, 
\[
\left\|\frac{u}{|x|}\right\|_{L^p}
\leq 
\frac{p}{N-p}
\left\|\nabla u\right\|_{L^p}.
\]
\end{theorem}
It clearly fails when $p \geq N$ since $|x|^{-p}$ is not locally integrable; however it  holds for $p \neq N$ for functions with compact support in $\R^N \setminus \{0\}$. Just write 
$|x|^{-p}=(N-p)^{-1}{\rm div}(x|x|^{-p})$, integrate by parts and use H\"older inequality.

 It is less known that Hardy inequality always holds for functions with  zero mean.
\begin{proposition}\label{prop:Hardy}
If $1 \leq p<\infty$, then there exists $C_p>0$ such that for every $u\in W^{1,p}_{ \geq 1}$, 
\[
\left\|\frac{u}{|x|}\right\|_{L^p}
\leq 
C_p
\left\|\nabla u\right\|_{L^p}.
\]
\end{proposition}
\begin{proof}
We use Poincar\`e inequality on the sphere $S^{N-1}$
$$\int_{S^{N-1}} |u(r\omega)|^p d\sigma(\omega) \leq C_p \int_{S^{N-1}} |\nabla_\tau u(r\omega)|^p d\sigma(\omega),$$
where $\nabla_\tau$ is the tangential gradient, since the function $u$ has zero mean for every fixed $r>0$. Since $|\nabla u|^2=u_r^2+r^{-2}|\nabla_\tau u|^2$ we get 
$$\int_{S^{N-1}} \frac{|u(r\omega)|^p}{r^p} d\sigma(\omega) \leq C_p \int_{S^{N-1}} |\nabla u(r\omega)|^p d\sigma(\omega),$$
and now it is sufficient to multiply both sides by $r^{N-1}$ and integrate it over $(0, \infty)$. \end{proof}

We consider also weaker versions of Hardy inequalities with the power $|x|^{-\alpha}$, $0 <\alpha \leq 1$.

\begin{lemma} \label{scaling}
If for $0<\alpha \leq 1$ the inequality 
$\left \|\frac{u}{|x|^\alpha}\right \|_p \leq C(\|u\|_p+\|\nabla u\|_p)$ holds in $W^{1,p}$, then the multiplicative inequality
$$\left \|\frac{u}{|x|^\alpha}\right \|_p \leq C\|u\|_p^{1-\alpha}\|\nabla u\|^\alpha_p$$
 actually holds.
\end{lemma}
{\sc Proof.} Just apply the hypothesis to $u_\lambda (x)=u(\lambda x)$ to get 
$$\left \|\frac{u}{|x|^\alpha}\right \|_p \leq C(\lambda^{-\alpha}\|u\|_p+\lambda^{1-\alpha}\|\nabla u\|_p)$$ and then minimize over $\lambda>0$.
\qed

The same scaling as above shows that no Hardy-type inequality can hold if $\alpha>1$ on a subspace of $W^{1,p}$ which is invariant under dilations.
\begin{proposition} \label{scaling}
If for $0<\alpha \leq 1$ the inequality 
$$\left \|\frac{u}{|x|^\alpha}\right \|_p \leq C\|u\|_p^{1-\alpha}\|\nabla u\|^\alpha_p$$
holds in $W^{1,p}$ if and only if $\alpha p <N$.
\end{proposition}
\begin{proof}
The condition $\alpha p<N$ is necessary since, otherwise, $|x|^{-\alpha p}$ is not locally integrable. Let us then assume it. If $p<N$ we use the classical Hardy inequality to have (here $B$ is the unit ball)
$$
\||x|^{-\alpha} u\|_p \leq \||x|^{-\alpha} u\|_{L^p(B)}+\||x|^{-\alpha} u\|_{L^p(B^c)} \leq \||x|^{-1} u\|_{L^p(B)}+\|u\|_p \leq C(\|\nabla u \|_p +\|u\|_p).
$$
If $N<p<\frac N \alpha$ we estimate the term $\||x|^{-\alpha} u\|_{L^p(B)}$ by $\|u\|_\infty \||x|^{-\alpha} \|_{L^p(B)} \leq C\|u\|_{W^{1,p}}$, by Sobolev embedding and, if $p=N$ we do the same with a large exponent $q$ such that $\alpha p \left (\frac{q}{p} \right )'<N$. 

In all cases we obtain $\||x|^{-\alpha} u\|_p \leq C\|u\|_{W^{1,p}}$ and conclude by the previous lemma.
\end{proof}

\section{Rellich inequalities}

Okazawa \cite{Okazawa1996} and  Davies--Hinz \cite{DaviesHinz1998}  proved  the following Rellich inequalities in $L^p$.

\begin{theorem}\label{lem:Rellich}
If $1<p<\frac{N}{2}$, then for every $u\in W^{2,p}(\R^N)$, 
\[
\left\|\frac{u}{|x|^2}\right\|_{L^p}
\leq 
\frac{p^2}{N(p-1)(N-2p)}
\left\|\Delta u\right\|_{L^p}.
\]
\end{theorem}

The condition $p< \frac N 2$ is necessary for the local integrability   of $|x|^{-2p}$. As in the case of Hardy inequality, one can investigate when Rellich inequality holds for functions with compact support in $\R^N \setminus \{0\}$ and having special symmetries.

\begin{proposition} \label{Rellichspecial}
Rellich inequalities
\[
\left\|\frac{u}{|x|^2}\right\|_{L^p}
\leq 
C_p
\left\|\Delta u\right\|_{L^p}
\]
hold in $C^\infty_c(\R^N \setminus \{0\})$ for $p \neq 1, \frac N 2, N$. 

When $p=1, \frac N 2$ they fail on radial functions but hold on $C^\infty_c(\R^N \setminus \{0\})\cap L^p_{\geq 1}$ and if $p=N$ they fail on $L^p_{=1}$ but hold on $C^\infty_c(\R^N \setminus \{0\})\cap L^p_{\neq 1}$.
\end{proposition} 
We refer to \cite[Section 6]{MSS-Rellich-disconti} and to \cite[Section 7]{MSS-Rellich-CZ} for the proof.

\begin{corollary} \label{Rellichalways}
Rellich inequalities
\[
\left\|\frac{u}{|x|^2}\right\|_{L^p}
\leq 
C_p
\left\|\Delta u\right\|_{L^p}
\]
hold in $W^{2,p}_{\geq 2}$ for $1 \leq p<\infty$. 
\end{corollary}
\begin{proof}
In fact they hold in $C^\infty_c(\R^N \setminus \{0\})\cap L^p_{\geq 2}$, by the previous proposition, and this set  is dense in $W^{2,p}_{\geq 2}$,  by Proposition \ref{density1}. 
\end{proof}

Next we consider Rellich inequalities with a power $|x|^{-\alpha}$. As for Hardy inequality one sees that the additive inequality 
$\left \|\frac{u}{|x|^\alpha}\right \|_p \leq C(\|u\|_p+\|\Delta u\|_p)$  implies the  multiplicative version
$$\left \|\frac{u}{|x|^\alpha}\right \|_p \leq C\|u\|_p^{1-\frac \alpha 2}\|\Delta u\|^{\frac \alpha 2}_p$$
 on any subspace of $W^{2,p}$ invariant under dilations. The above corollary shows that this is the case on $W^{2,p}_{\geq 2}$ for any $0<\alpha \leq 2$. Since the case $\alpha=2$ has been already considered above, we treat only $0<\alpha<2$.

\begin{proposition} \label{Rellich}
Let $1<p<\infty$,  $0<\alpha <2$. Then Rellich inequalities
\[
\left\|\frac{u}{|x|^\alpha}\right\|_{L^p}
\leq 
C\|u\|_{L^p}^{1-\frac \alpha 2}\|\Delta u\|_{L^p}^{\frac \alpha 2}.
\]
hold in $W^{2,p}_{\geq n}$, $n=0,1$,  if and only if $\alpha < \frac{N}{p}+n$. 
\end{proposition}
\begin{proof}
The necessity of the conditions follow since $W^{2,p}_{\geq n}$ contains functions which behave like $|x|^n$ near the origin, which forces $|x|^{(n-\alpha)p}$ to be locally integrable.

As explained above, it is sufficient to prove the additive inequality $\||x|^{-\alpha}u\|_p \leq C(\|u\|_p+\|\Delta u\|_p)$, $u \in W^{2,p}_{\geq n}$. Also, since $1<p<\infty$, we can reduce the proof to showing that  $\||x|^{-\alpha}u\|_p \leq C\|u\|_{W^{2,p}}$, by using Calder\'on-Zygmund inequality.

{\bf Case 1.} $n=0, 2 < \frac N p.$ This case is already covered by Lemma \ref{Rellich} which holds with $\alpha=2$, splitting the integrals over and outside the unit ball, as in the proof of Hardy's inequalities.

{\bf Case 2.} $n=0, \alpha < \frac N p \leq 2.$ Assume first that $\frac N p < 2$ so that, by Morrey embedding, $\|u\|_\infty \leq C\|u\|_{W^{2,p}}$. Then 
$$
\||x|^{-\alpha} u\|_p \leq \||x|^{-\alpha} u\|_{L^p(B)}+\||x|^{-\alpha} u\|_{L^p(B^c)} \leq \|u\|_\infty \||x|^{-\alpha} \|_{L^p(B)}+\|u\|_p \leq C\|u\|_{W^{2,p}}.
$$
If $\frac N p=2$ one has to use a large exponent $q$ instead of $\infty$ and use H\"older inequality in $B$.

The case $n=0$ is concluded and we consider $n=1$. We may assume that $\frac N p \leq \alpha < \frac N p+1$, otherwise we are again in cases 1 or 2.

{\bf Case 3.} $n=1, \frac N p  \leq \alpha < \frac N p+1 $. Note that $\frac N p \leq \alpha <2$, so that $p> \frac N 2$. If $p < N$, then Morrey embedding gives that $u$ is H\"older continuous of exponent $\gamma= 2-\frac Np$. However, since $u(r\dot)$ has zero mean for every $r>0$, then $u(0)=0$ and $|u(x)| \leq C\|u\|_{W^{2,p}}|x|^\gamma$ for $|x| \leq 1$. Proceeding as in case 2
$$
\||x|^{-\alpha} u\|_p \leq \||x|^{-\alpha} u\|_{L^p(B)}+\||x|^{-\alpha} u\|_{L^p(B^c)} \leq C\|u\|_{W^{2,p}} \||x|^{\gamma-\alpha} \|_{L^p(B)}+\|u\|_p \leq C\|u\|_{W^{2,p}}
$$
since $p(\alpha-\gamma)=p(\alpha-2)+N<N$. 

When $p=N$,   $u$ is H\"older continuous of any exponent less than 1 and we repeat the proof above. 

Finally, if $p>N$,  $u \in C^1$ and we use the estimate $|u(x)| \leq C\|u\|_{W^{2,p}}|x|$ for $|x| \leq 1$.
\end{proof} 

\begin{corollary} \label{Rellich1}
If $0< \alpha <2$ and $\frac N p \leq \alpha < \frac N p+1$, then Rellich inequalities as in Proposition \ref{Rellich} hold in $\{u \in W^{2,p}: u(0)=0\}$.
\end{corollary}
\begin{proof}
Exactly as in Case 3 of the above proposition (note that $p> \frac N 2$)
\end{proof}
Arguing similarly (note that $p>N$ below) one obtains
\begin{corollary} \label{Rellich2}
If $0< \alpha <2$ and $\frac N p +1 \leq \alpha < 2$, then Rellich inequalities as in Proposition \ref{Rellich} hold in $\{u \in W^{2,p}: u(0)=\nabla u(0)=0\}$.
\end{corollary}
\section{The operator in $L^2$}
The easiest way to define the operator  is through a form in $L^2$.
For $0<\alpha <2$, $N \geq 2$,  we introduce the symmetric form on $H^1=W^{1,2}$ 
$$
a(u,v)=\int_{\R^N}\left ( \nabla u\cdot \nabla \bar{v} +c \frac{u \bar v}{|x|^\alpha}\right )\, dx.
$$
By Proposition \ref{scaling} with $p=2$ the form $a$ is well-defined, continuous on $H^1$ and bounded from below (use $\||x|^{-\frac{\alpha}{2}} u\|_2 \leq \varepsilon \|\nabla u\|_2+C_\varepsilon \|u\|_2$). We can therefore define a selfadjoint operator $S$ in $ L^2$, bounded from below,  by
\begin{eqnarray} \label{Sform}
D(S)=\{ u \in H^1: \exists  f \in L^2\ {\rm such\ that}\  a(u,v)=\int_{\R^N} f\bar v\, dx\ \forall  v \in H^1\}, \quad  Su=f.
\end{eqnarray}
The generated semigroup $\{e^{-z S}\}$ is analytic for ${\rm Re}\,  z >0$ in $L^2$ and  positive for $t \geq 0$, since $a(u^+, u^-)=0$. 

We refer to \cite[Chapters 1,2]{Ouhabaz} for the basic properties of operators and semigroups associated to forms.

The semigroup is not $L^\infty$-contractive unless $c \geq 0$ but we could use Gaussian estimates to extend it to $L^p$. 

However, we follow another strategy in the next sections, which gives the domain. We define the operator directly in $L^p$ when $0<\alpha < \frac Np+1$, the easiest case being $\alpha < \frac N p$ , and show that it generates an analytic semigroup. When $p=2$ we prove that this operator coincides with $S$ defined in \eqref{Sform} and that all these semigroups are consistent for different values of (admissible) $p$. In particular, we characterize the domain in $L^2$ of the operator $S$.
Finally, since the semigroup is selfadjoint, by duality it is also a semigroup when $\frac Np+1 \leq \alpha <2$ and we characterize its domain  using elliptic regularity.

\section{The operator in  $L^p$}
\subsection{The case $\alpha < \frac N p+1$}
Let us start with the simplest result.

\begin{proposition} \label{generation1}
If $0 < \alpha < 2$, $1<p<\infty$, $ \alpha<\frac N p$, then $-S$ with domain $W^{2,p}$ generates an  analytic semigroup of angle $\pi/2$ in $L^p$. 
\end{proposition}
\begin{proof}
Proposition \ref{Rellich} for $n=0$ gives 
$\|x|^{-\alpha}|u\|_p \leq C\|u\|_p^{1- \frac \alpha 2}\|\Delta u\|_p^{\frac \alpha 2}$ 
for $u\in W^{2,p}$, and then, 
for every $\varepsilon >0$ there exists a positive constant $C_\varepsilon$ such that  
$\|x|^{-\alpha}|u\|_p \leq  \varepsilon \|\Delta u\|_p+C_\varepsilon\| u\|_p$. The result follows from standard perturbation theory for analytic semigroups.
\end{proof}

\medskip

The following lemma is crucial to treat other cases.
\begin{lemma}\label{lem:psi_alpha}
If $0<\alpha <2$ and, $N \geq 2$ ,
then the function 
\[
\psi_{\alpha,c,m}(x)=
\sum_{k=0}^m
\frac{c^k\Gamma(\frac{N-\alpha}{2-\alpha})}{(2-\alpha)^{2k} k!\Gamma(\frac{N-\alpha}{2-\alpha}+k)}
|x|^{(2-\alpha)k}
\]
satisfies
\[
\left(-\Delta+\frac{c}{|x|^\alpha}\right)\psi_{\alpha,c,m}(x)=
\frac{c^{m+1}\Gamma(\frac{N-\alpha}{2-\alpha})}{(2-\alpha)^{2m}m!\Gamma(\frac{N-\alpha}{2-\alpha}+m)}
|x|^{(2-\alpha)m-\alpha}.
\]
In particular, if $m\geq \frac{\alpha}{2-\alpha}$, then $\left(-\Delta+\frac{c}{|x|^\alpha}\right)\psi_{\alpha,m}\in L^\infty(B)$. 
\end{lemma}
\begin{proof}
By a direct calculation, we have for $k \geq 1$ and with $\xi(x)=|x|^{2-\alpha}$
\[
\Delta \xi^k= \frac{k(2-\alpha)(N-2+k(2-\alpha))}{|x|^\alpha} \xi^{k-1}:= \frac{\beta_k\xi^{k-1}}{|x|^\alpha}.
\]
If $\psi_{\alpha,m}(x)=\sum_{k=0}^m \gamma_k \xi^k$ with $\gamma_0=1$, then 
\[\left(-\Delta +\frac{c}{|x|^{\alpha}}\right)\psi_{\alpha,m}(x)
=\frac{1}{|x|^{\alpha}}
\left(\sum_{k=0}^{m-1} (c\gamma_k-\beta_{k+1}\gamma_{k+1})\xi^k\right)
+c\gamma_m |x|^{(2-\alpha)m-\alpha}.
\]
Requiring  that the lower order terms vanish, we get  for $k<m$ 
\[
\frac{\gamma_{k+1}}{\gamma_k}
=\frac{c}{\beta_{k+1}}
=\frac{c}{(2-\alpha)^2 (k+1) \left (\frac{N-\alpha}{2-\alpha}+k\right )}
\]
and the formula in the statement follows.
\end{proof}

Using the behavior of $\psi_{\alpha,c,m}$ in a neighborhood of the origin, 
we define the following auxiliary function and related multiplication operator.
\begin{definition}\label{def:phi}
{\bf (i)}
For $0< \alpha<2$ and $c\in\R$, we fix a function $\phi=\phi_{\alpha,c}\in C^\infty (\R^N \setminus \{0\})$ satisfying 
\begin{itemize}
\item[(i-1)] $\phi$ is radial and $\frac{1}{2}\leq \phi\leq 2$ on $\R^N$, 
\item[(i-2)] $\phi\equiv \psi_{\alpha,c,m}$ with $m \in [\frac{\alpha}{2-\alpha},\frac{2}{2-\alpha})$
in the neighbourhood of the origin, 
\item[(i-3)] $\phi(x)\equiv 1$ in a neighbourhood of  infinity. 
\end{itemize}
{\bf (ii)} 
We define the multiplication operator $T: L^p \to L^p$, $Tf=\phi f$ which is bijective from $L^p$ to itself. If $u \in W^{2,p}(\R^N \setminus B_\epsilon)$, for every $\epsilon >0$ we have
\begin{equation} \label{identity}
T^{-1}STu=T^{-1}(-\Delta +c|x|^{-\alpha}) Tu=-\Delta u-2 \frac{\nabla \phi}{\phi} \cdot \nabla u + Vu:= \tilde S u 
\end{equation}
with $V=\frac{\Delta \phi-c|x|^{-\alpha} \phi}{\phi}$ bounded in  $\R^N$.
\end{definition}

Observe that $\phi$ is a polynomial in $|x|^{2-\alpha}$ near the origin, $\phi(x)=1+\kappa |x|^{2-\alpha}+\cdots, $ with  $\kappa=\frac{c}{(2-\alpha)(N-\alpha)}$. In particular $|\nabla \phi| \approx |x|^{1-\alpha}, D_{ij} \phi \approx |x|^{-\alpha}$ near 0,  so that $\phi \in W^{2,p}$ if and only if $\alpha <\frac N p$ and $\phi \in W^{1,p}$ if and only if $\alpha < \frac N p+1$. Finally, $\nabla \phi$ is bounded when $\alpha \leq 1$. Similar remarks hold for $\phi^{-1}$.

\begin{proposition} \label{generation2}
If $0<\alpha<2$ , $N \geq 2$, $1<p<\infty$, $0<\alpha < \frac N p+1 $, then $-S$ with domain $$D(-S)=\{ u \in W^{2,p}(\R^N \setminus B_\epsilon) {\rm \  for\  every\ }\epsilon >0, \frac{ u}{\phi} \in W^{2,p}\}$$ generates an  analytic semigroup of angle $\pi/2$ in $L^p$. 
\end{proposition}
\begin{proof}
From \eqref{identity} we have $S=T\tilde S T^{-1}$ and,  $\phi^{-1}\nabla \phi \approx |x|^{1-\alpha}$ near the origin and has compact support. 

The operator $\tilde S_0=-\Delta +V$ is uniformly elliptic with bounded coefficients and hence generates an analytic semigroup $\{e^{- z\tilde {S_0}} \}$  of angle $\frac {\pi}{ 2}$ when endowed with the domain $W^{2,p}$.

Next we consider $\tilde S=\tilde S_0-2 \frac{\nabla \phi}{\phi}$. By Hardy inequality in Proposition \ref{scaling}, since $(\alpha-1)p <N$, $$\|\phi^{-1}\nabla \phi \nabla u \|_p \leq C\|\nabla u\|_p^{1-\alpha} \|D^2 u\|_p^{\alpha} \leq \epsilon \|D^2 u\|_p +C_\epsilon \|\nabla u\|_p   \leq \epsilon \|\Delta  u\|_p +C_\epsilon \|u\|_p,$$
by Calder\'on-Zygmund inequality and standard interpolation inequalities. It follows that the term $\phi^{-1} \nabla \phi \nabla u $ is a small perturbation of $-\Delta$, hence of $-\tilde {S_0}$ and then $-\tilde S$  with domain $W^{2,p}$ generates an analytic semigroup of angle $\frac{\pi}{2}$, by standard perturbation theory of analytic semigroups.

 It follows that $e^{-z S}=Te^{-z \tilde S}T^{-1}$ is analytic in the same region. Finally, $u \in D(-S)$ if and only if $T^{-1} u=\phi^{-1} u \in D(-\tilde S)=W^{2,p}$. This concludes the proof, since both $\phi, \phi^{-1}$ are smooth out of the origin, hence preserve $W^{2,p} (\R^N \setminus B_\epsilon)$. 
\end{proof} 

In the following proposition we characterize $D(-S)$. In particular we prove that, when $\alpha < \frac N p$, the semigroups constructed in Propositions \ref{generation1}, \ref{generation2} coincide

\begin{theorem} \label{domain1}
If $0<\alpha<2$, $N \geq 2$, $1<p<\infty$, $0<\alpha < \frac N p+1$. Let us fix $\eta \in C_c^\infty (\R^N)$, $ \eta \equiv 1$ in $B$. Then
\begin{itemize}
\item[(i)]  if $\alpha < \frac N p$, then $D(-S)=W^{2,p}$;
\item[(ii)] if $ \frac N p \leq \alpha< \frac Np+1$, then $D(-S) \subset W^{1,p}$ and $$D(-S)=\{u=u(0) \eta \phi+u_1, \  u_1 \in W^{2,p}, u_1(0)=0\}.$$
\end{itemize}
\end{theorem}
\begin{proof}
For (ii) note that $p >\frac N 2$. We pick $u \in D(-S)$ and let $v=\phi^{-1} u \in W^{2,p}$. We split $v=v(0)\eta+v_1$ with $v_1=v-v(0)\eta \in W^{2,p}$ and $v_1(0)=0$. Since  $\alpha < \frac N p +1$ Corollary \ref{Rellich1} and Proposition \ref{scaling} yield $\frac {v_1}{ |x|^{\alpha}}, \frac{\nabla v_1}{|x|^{\alpha-1}} \in L^p$. Also  $|\nabla \phi| \approx |x|^{1-\alpha}, |D^2 \phi| \approx |x|^{-\alpha} $ and  this easily implies that $u_1=\phi v_1 \in W^{2,p}$ and the representation follows.
Conversely, if $u$ has the above form, the same argument gives $\phi^{-1} u \in W^{2,p}$.

The proof of (i) is similar but simpler, without splitting the function $v$.
\end{proof} 

\begin{remark}
Note that, if $\alpha \leq 1$, then $\phi(x)=1+\kappa |x|^{2-\alpha}, \kappa=\frac{c}{(2-\alpha)(N-\alpha)}$.
Note also that, when $\frac N p \leq \alpha < \frac Np+1$, then $D(-S)$ depends on the constant $c$, since this happens for $\phi$.
\end{remark}

We denote by $S_p$, $e^{-z S_p}$ the operator and the semigroup of the above proposition and show consistency for different $p$. In particular, we prove that they coincide with  $S$, $e^{-z S}$ of Section 5, defined by form methods in $L^2$. 

\begin{proposition} \label{consistency} Under the hypotheses of  Theorem \ref{domain1} we have $e^{-z S_p}f=e^{-z S f}$ and $(\lambda +S_p)^{-1}f=(\lambda +S)^{-1}f$ for $f \in L^p \cap L^2$ and $\lambda$ big enough. In particular the semigroups and the resolvent are consistent for different (admissible) values of $p$ and $e^{-t S_p}$ is positive.
\end{proposition}
\begin{proof}
First we show that the semigroups are consistent for different values of $p$. Keeping the notation of Proposition \ref{generation2}, we have in fact $e^{-z S_p}=Te^{-z \tilde{S}}T^{-1}$. The map $T$ is independent of $p$ and the same holds for the  semigroup $e^{-z \tilde{S}}$, since $\tilde S=\tilde{S_0}-2 \phi^{-1} \nabla \phi$ is a small perturbation of the uniformly elliptic operator $\tilde{S_0}$.

Finally, let us show that $S_2=S$ in $L^2$. We first note that $\alpha< \frac N 2+1$, since $N \geq 2$ and $\alpha <2$, so that Proposition \ref{generation2} applies with $p=2$. It is sufficient to prove that $S$ is an extension of $S_2$, since both operators generate a semigroup.

Let $u \in D(S_2) \subset H^1, S_2u=f \in L^2$. 

If $ \alpha < \frac N 2$, then $u \in H^2$, by Theorem \ref{domain1} and, integrating by parts, $a(u,v)=\int_{\R^N} fv$ for every $v \in H^1$, so that $u \in D(S)$ and $Su=f$.

Finally, assume that $\frac N 2 \leq \alpha <  \frac N 2+1$ and write $u=u(0) \phi \eta+ u_1$ with $u_1 \in H^2$. For $u_1$ we may integrate by parts obtaining $a(u_1,v)=\int_{\R^N} (S_2 u_1) v$ for every $v \in H^1$. It remains to show that the same holds for $w=\phi\eta$ which is not in $H^2$. Let us fix a ball $B_R$ containing the support of $\eta$ and $v \in C_c^\infty (\R^N)$. Recalling that $S_2 w$ is bounded we get 
\begin{align*}
\int_{\R^N} (S_2 w) v&=\lim_{\epsilon \to 0} \int_{B_R \setminus B_\epsilon} \left (-\Delta w+ c|x|^{-\alpha}w\right ) v \\
&=\lim_{\epsilon \to 0} \left(\int_{B_R \setminus B_\epsilon} \nabla w \cdot \nabla v+\int_{\partial B_\epsilon} v \frac{\partial w}{\partial n}\, d\sigma \right )+\int_{\R^N} c|x|^{-\alpha} wv \\
&=\int_{\R^N} \left(\nabla w\cdot \nabla v +c|x|^{-\alpha} wv\right)
\end{align*}
since $|\nabla w| \leq C \epsilon^{1-\alpha}$ on $\partial B_\epsilon$ and $v$ is bounded. This shows that $a(w,v)= \int_{\R^N} (S_2w) v$ for all $v \in C_c^\infty$ and, by density, for all $v \in H^1$.
\end{proof} 

We can finally prove that the semigroup $e^{-z S}$ consists of bounded operators in $L^p$ for all $1<p<\infty$ and it is strongly continuous.

\begin{corollary} \label{allalpha}
If $0<\alpha <2$, $N \geq 2$, then $\{e^{-zS}\}$ extrapolates to an analytic semigroup of angle $\frac{\pi}{2}$ in $L^p$ for every $1<p<\infty$.
\end{corollary}
\begin{proof} In fact $\{e^{-z S}\}$ extrapolates to $\{e^{-z S_p}\}$ if $\alpha < \frac N p+1$, in particular this happens if $ 1<p \leq 2 \leq N$. Since $S$ is selfadjoint, $\{e^{-z S}\}$ is analytic in $L^{p'}$ whenever $\{e^{-z S}\}$ is so in $L^p$, and this completes the proof.
\end{proof}

\subsection{The case $ \frac N p+1 \leq \alpha <2$}

We need the following lemma which is the companion to Lemma \ref{lem:psi_alpha}.
\begin{lemma}\label{lem:psitilde_alpha}
For $1<\alpha <2$, 
define the function $\widetilde{\psi}_{\alpha,c,m}$ as 
\[
\widetilde{\psi}_{\alpha,c,m}(x)=
\sum_{k=0}^m
\frac{c^k\Gamma(\frac{N}{2-\alpha}+1)}{(2-\alpha)^{2k}k!\Gamma(\frac{N}{2-\alpha}+1+k)}
|x|^{(2-\alpha)k}.
\]
Then  
\[
\left(-\Delta+\frac{c}{|x|^\alpha}\right)\big(x_j\widetilde{\psi}_{\alpha,c,m}\big)=
Cx_j|x|^{(2-\alpha)m-\alpha}.
\]
In particular, if $m\geq \frac{\alpha-1}{2-\alpha}$, then $\left(-\Delta+\frac{c}{|x|^\alpha}\right)\big(x_j\widetilde{\psi}_{\alpha,m}\big)\in L^\infty(B)$. 
\end{lemma}
\begin{proof}
Put $\xi(x)=|x|^{2-\alpha}$. 
It suffices to observe 
\[
\Delta\left(x_j\xi^{k}\right)
=
k(2-\alpha)\big(N+k(2-\alpha)\big)\frac{x_j\xi^{k-1}}{|x|^{\alpha}}.
\]
The rest is the same as  in the 
proof of Lemma \ref{lem:psi_alpha}.
\end{proof}
\begin{definition}\label{def:phij}
For $0< \alpha<2$ and $c\in\R$, we fix a function $\phi_j\in C^\infty (\R^N \setminus \{0\})$ satisfying 
\begin{itemize}
\item[(i-1)] $\phi_j= x_j\widetilde{\psi}_{\alpha,c,m}(x)$ with $m \in [\frac{\alpha-1}{2-\alpha},\frac{1}{2-\alpha})$
in the neighbourhood of the origin, 
\item[(i-2)] $\phi_j\equiv 0$ in a neighbourhood of  infinity. 
\end{itemize}
\end{definition}
Note that $\phi_j(x)=x_j(1+c_1 |x|^{2-\alpha}+\dots)$ near zero.

\begin{lemma} \label{ultraweak} Let  $N \geq 2$ and $w$ be the function $\eta \phi$ with $\phi$ as in Definition \ref{def:phi} and  $\eta \in C_c^\infty$ or one of the functions $\phi_j$ of Definition \ref{def:phij}. Then for every $v \in C_c^\infty$
$$
\int_{\R^N} w(-\Delta v+\frac{c}{|x|^{\alpha}} v)=\int_{\R^N} fv, \qquad f =-\Delta w+\frac{c}{|x|^{\alpha}} w \in L^\infty {\ \rm with\ compact\ support}.
$$ 
\end{lemma}
{\sc Proof.} We write the difference between the left and right hand side as
$$\lim_{\epsilon \to 0} \int_{\R^N \setminus B_\epsilon} \left (v\Delta w-w \Delta v\right )= \lim_{\epsilon \to 0} \int_{\partial B_\epsilon}\left (w \frac{\partial v}{\partial n}-v\frac{\partial w}{\partial n} \right ) d\sigma=0
$$
since $w$ is bounded near the origin and $\nabla w$ grows at most as $|x|^{1-\alpha}$ as $x \to 0$.

\qed

By Corollary \ref{allalpha} we may introduce $-S_p$ the generator of the extrapolated semigroup in $L^p$ for every $1<p<\infty$. Clearly, if $\alpha < \frac N p+1$, then $-S_p=-\Delta + \frac{c}{|x|^{\alpha}}$ on $D(S_p)$ given by Theorem \ref{domain1}.

In the following theorem we use the functions $\phi, \phi_j$ of Definitions \ref{def:phi} , \ref{def:phij} and fix $\eta \in C_c^\infty$ such that $\eta \equiv 1$ in $B$.

\begin{theorem} \label{domain2}
Assume that $N \geq 3$ and $\frac N p+1 \leq \alpha <2$. Then 
\begin{equation} \label{domainplarge}
D(S_p)=W_0^{2,p}\oplus {\rm span}\{\eta\phi, \phi_1,\ldots,\phi_N\}.
\end{equation}
\end{theorem}
\begin{proof} 
By construction, 
we recall that $S_p$ is the adjoint of $S_{p'}$. Note that  $p>N >\frac{N}{N-\alpha}$ since  $N \geq 3$, so that  $p' < \frac{N}{\alpha}$ and, by Theorem \ref{domain1}, $D(S_{p'})=W^{2,p'}$.

Let $W$ be the right hand side of \eqref{domainplarge}. By Corollary \ref{Rellich2} and Lemma \ref{ultraweak}, if $u \in W$ and $v \in C_c^\infty$
$$
\int_{\R^N} u(-\Delta v+\frac{c}{|x|^{\alpha}} v)=\int_{\R^N} fv, \qquad f =-\Delta u+\frac{c}{|x|^{\alpha}} u \in L^p.
$$ 
By density, using Proposition \ref{Rellich}, the above equality
can be extended to any $v \in W^{2,p'}=D(S_{p'})$, 
so that $u \in D(S_{p'}^*)=D(S_p)$ and $S_p u =f$. 

The converse inclusion is longer and will be done in steps. Assume that $u \in D(S_p)=D(S_{p'}^*)$. This means that there exists $f \in L^p$ such that
\begin{equation} \label{weak}
\int_{\R^N} u\left (-\Delta v+ \frac{c}{|x|^\alpha} v \right )=\int_{\R^N} fv
\end{equation}
for every $v \in W^{2,p'}$ and $S_pu=f$. By elliptic regularity for the Laplacian,  see  \cite[Lemma 5.1]{Agmon1959},  we obtain that $u \in W^{2,p}(\R^N \setminus B_\epsilon)$ for every $\epsilon >0$ and $-\Delta u +\frac{c}{|x|^\alpha } u=f$.

\smallskip

{\bf Step 1.} By the H\"older inequality, we have
$$
\left\|\frac{u}{|x|^\alpha}\right\|_{q}^q=
\int_B \frac{|u|^q}{|x|^{\alpha q}} \leq \left (\int_B |u|^p \right )^{\frac  qp} \left (\int_B \frac{1}{|x|^{\frac{\alpha pq}{p-q}} }\right )^{1-\frac q p} \leq C\|u\|_p^q
$$
with $C<\infty$ when $q < \frac{Np}{N+\alpha p} <p$, or 
equivalently, $\frac 1 q > \frac 1p+\frac \alpha N$. 
Rewriting \eqref{weak} as
\begin{equation} \label{weak1}
-\int_{\R^N} u\Delta v=\int_{\R^N}\left ( f-\frac{c}{|x|^\alpha} u \right )v
\end{equation}
for every $v \in W^{2,p'}$ and using local elliptic regularity for the Laplacian, we obtain that $u \in W^{2,q}(B_r)$ for any $r<1$.

\smallskip

{\bf Step 2.} We now iterate step 1. Choose $q< \frac{Np}{N+\alpha p}$ and close to it. 

If $q \geq \frac N2$, then $u \in L^s (B_r)$ for every $s <\infty$, by Sobolev embedding and then $|x|^{-\alpha} u \in L^{\frac N \alpha -\epsilon}$ for every $\epsilon >0$.

If $q < \frac N2$, then by Sobolev embedding again, $u \in L^{q_1}(B_r)$ for $\frac{1}{q_1}= \frac 1q -\frac 2 N$ which we can make close as we want to 
$\frac 1 p +\frac{\alpha-2}{N}$. By iterating (if $q_n < \frac N 2$), $u \in L^{q_{n+1}}(B_r)$ if $\frac{1}{q_{n+1}}=\frac{1}{q_n}-\frac 2N\approx \frac 1p + \frac{n(\alpha-2)}{N}$.
In a finite number of steps, $u \in L^\infty (B_r)$ and then $|x|^{-\alpha} u \in L^{\frac N \alpha -\epsilon}(B_r)$ for any $\epsilon >0$, as above.
Then, by \eqref{weak1}, $u \in W^{2, \frac N \alpha -\epsilon}(B_r) $ for any $\epsilon >0$.

\smallskip

{\bf Step 3.} Choose $\epsilon$ small enough so that $\frac N \alpha -\epsilon > \frac N 2$. By step 2 and Sobolev embedding, $u$ 
is H\"older continuous with the exponent $2-\alpha-\epsilon$, hence $|u(x)-u(0)| \leq C|x|^{2-\alpha-\epsilon}$ for $|x| \leq r$. Let us consider $u_1=u-u(0) \phi \eta$ where $\phi$ is as in Definition \ref{def:phi} and $\eta \in C_c^\infty$, $\eta \equiv 1$ in $B$. Since $\phi(x)=1+\kappa  |x|^{2-\alpha}+\dots$ we get $|u_1(x)| \leq C|x|^{2-\alpha-\epsilon}$ in $B_r$.
By Lemma \ref{ultraweak} the function $\eta\phi$ satisfies \eqref{weak1} with $f \in L^\infty$ with a compact support. By difference, $u_1$ satisfies \eqref{weak1} for some $f \in L^p$.

However, $|x|^{-\alpha} |u_1| \leq \frac{C}{ |x|^{2\alpha-2+\epsilon}} \in L^q(B_r)$ if $ q <\frac{N}{2\alpha-2}$. This last exponent is bigger than $\frac N \alpha$ but smaller than $p$. We can use elliptic regularity again and deduce that $u_1 \in W^{2, \frac{N}{2\alpha-2}+\epsilon}(B_r)$ and hence $u_1$ is H\"older continuous of exponent $4-2\alpha-\epsilon$ if $4-2\alpha \leq 1$. If, instead $4-2\alpha >1$ we get immediately $|u_1(x)| \leq C|x|$.

In the first case we get  $|x|^{-\alpha}|u_1(x)| \leq \frac{C}{ |x|^{3\alpha -4+\epsilon}}$ in $B_r$ and repeat the argument above to obtain H\"older continuity of exponent $(6-3\alpha +\epsilon) \wedge 1$. In a finite number of steps we get H\"older continuity of exponent $(k(2-\alpha) +\epsilon) \wedge 1$ and hence 1, so that $|u_1(x) |\leq C|x|$ in $B_r$.

{\bf Step 4.} $\frac{|u_1|}{|x|^\alpha} \leq \frac{C}{|x|^{\alpha -1}} \in L^q(B_r)$ for any $q< \frac{N}{\alpha-1}$ and then, by elliptic regularity, $u_1 \in W^{2, q}(B_r)$. Since $\frac{N}{\alpha-1} >N$ we obtain that $u_1 \in C^{1, \gamma}(B_r)$ for any $\gamma <1-(\alpha-1)=2-\alpha$. 
Let us define 
$$u_2(x)=
u_1(x)-\sum_{j=1}^N D_j u_1(0) \phi_j(x)
=
u_1(x)-\sum_{j=1}^N D_j u_1(0) x_j\widetilde{\psi}(x),
$$ 
where the $\phi_j$ are those from Definition \ref{def:phij}. As in Step 3, the function $u_2$ satisfies \eqref{weak1} (using again Lemma \ref{ultraweak}) and, moreover, $|u_2(x)| \leq C |x|^{3-\alpha-\epsilon}$ in $B_r$. At this point the same iteration as in Step 3 shows that $u_2 \in W^{2,p}$ and then 
$u_2 \in W_0^{2,p}$ since $u_2(0)=\nabla u_2(0)=0$
(this time the iteration ends when $|x|^{-\alpha} u_2 \in L^p(B_r)$ since in the right hand side of \eqref{weak1}, $f \in L^p$).
\end{proof}

\begin{remark}
The proof shows that if $u=c_0 \eta \phi+\sum_{j=1}^N c_j \phi_j +v \in D(S_p)$, then $u$ is continuous and $c_0=u(0)$. Moreover, $u_1=u-u(0)\eta \phi$ is continuously differentiable and $c_j=D_j u_1(0)$, $j=1, \dots , N$.
\end{remark}

If $N=2$ the inequality $N >\frac{N}{N-\alpha}$ fails but the above proof still works when $p>\frac{2}{2-\alpha}$. Since, by assumption, $\frac 2p +1 \leq \alpha$ we have $p \geq \frac{2}{\alpha -1}$. Therefore, if $\alpha < \frac 32$ then $\frac{2}{\alpha-1} > \frac {2}{2-\alpha}$ and the above theorem holds. Only the case $\frac 32 \leq \alpha <2$ and $ \frac{2}{\alpha-1}  \leq p \leq  \frac {2}{2-\alpha}$ is still missing.

\begin{theorem}
If $N=2$ and $\frac 2p +1 \leq \alpha <2$, then 
\[
D(S_p)=W^{2,p}_0\oplus {\rm span}\{\eta\phi,\phi_1, \phi_2\}.
\]
\end{theorem}
\begin{proof} As explained above, only the case $\frac 32 \leq \alpha <2$ and $ \frac{2}{\alpha-1}  \leq p \leq  \frac {2}{2-\alpha}$ requires a proof. Note also that $p \geq 4$.

Let $W$ be the right hand side in the statement and fix $p_1< p<p_2$ with $p_1 < \frac 2\alpha$ and $p_2 >\frac{2}{2-\alpha}$. If $w$ is one of the functions $\eta \phi, \phi_1, \phi_2$, then $w \in D(S_{p_2})$.  However $w \in W^{2,p_1} =D(S_{p_1})$, too.  This gives 
$$
\frac{e^{-tS} w-w}{t} \to -Sw, \quad t \to 0
$$
both in $L^{p_1}$ and $L^{p_2}$, hence in $L^p$. The same argument applies to any $w \in C_c^\infty (\R^N \setminus \{0\})$ and gives $$\left (C_c^\infty (\R^N \setminus \{0\}\right )\oplus {\rm span}\{\eta\phi,\phi_1, \phi_2\} \subset D(S_p).$$
By Proposition \ref{Rellich2} the multiplication by $|x|^{-\alpha}$ is a small perturbation of the Laplacian on $W^{2,p}_0$ and then the graph norm and the $W^{2,p}$ norm are equivalent on $C_c^\infty (\R^N \setminus \{0\})$. It follows that  $W \subset D(S_p)$ and that the graph norm and the $W^{2,p}$ are equivalent on $W^{2,p}_0$. Since ${\rm span}\{\eta\phi,\phi_1, \phi_2\} $ is finite dimensional, $W$ is closed in $D(S_p)$ for the graph norm and we have to show that it is dense.
Let us consider $$Z=W^{2,p_1} \cap \left (W^{2,p_2}_0\oplus {\rm span}\{\eta\phi,\phi_1, \phi_2\}\right )=D(S_{p_1}) \cap D(S_{p_2})\subset W.$$ 

To justify the last inclusion, take $u \in Z, u=v+w$ with $u \in W^{2,p_2}_0$ and $w \in {\rm span}\{\eta\phi,\phi_1, \phi_2\} $. Since $u,w \in W^{2,p_1}$, too, then $v \in W^{2,p_1}$ and hence in $W^{2,p}_0$.

$Z$ in dense in $L^p$ since contains $C^\infty_c (\R^N \setminus\{0\})$ and is invariant under the semigroup $e^{-tS}$. By the core theorem it is dense in $D(S_p)$ and this concludes the proof. 
\end{proof}

\section{Quasi-accretivity in $L^p$}
As already explained in the Introduction, the semigroup $e^{-tS}$ satisfies upper and lower Gaussian estimates. It is well know, then, that it extrapolates to an analytic semigroup in $L^1$ and that the spectrum is independent of $p$. In this section we concentrate on the simpler inequality $\|e^{-tS_p}\|_p \leq e^{\omega_p t} $ for $t \geq 0$. This inequality is clearly true with $\omega_p=0$, when $c \geq 0$, by domination with the heat semigroup and therefore we deal only with the case $c<0$, without aiming to compute the best constant $\omega_p$.

\begin{lemma}\label{lem:quasi-accretive}
Let  $ \alpha\in (0,2)$, $N \geq 2$. Then for every  $\epsilon>0$ there exists $C_\epsilon>0$ such that  for every $u \in C^\infty_c$
$$
\int_{\R^N} \frac{|u(x)|^p}{|x|^\alpha} \leq \epsilon \int_{\R^N} |\nabla u|^2|u|^{p-2} \chi_{\{u \neq 0\}} +C_\epsilon \int_{\R^N} |u|^p
$$
\end{lemma}
\begin{proof} From the identity
$$|u(x)|^p=p\int_1^\infty |u(tx)|^{p-2} u(tx) \nabla u(tx)\, dt$$ we obtain
$$\frac{|u(x)|^p}{|x|^\alpha} \leq p\int_1^\infty\frac{ |u(tx)|^{p-1}}{|x|^{\alpha-1}} | \nabla u(tx)|\, dt$$ and then (set $y=tx$) 
\begin{align} \label{dissip1}
\nonumber \int_{\R^N}|u(x)|^p \, dx &\leq p\int_1^\infty \int_{\R^N} \frac{|u(tx)|^{p-1}}{|x|^{\alpha-1}}  |\nabla u(tx)|\, dt= p \int_1^\infty \frac{dt}{t^{N-\alpha+1}}\int_{\R^N} \frac{|u(y)|^{p-1}}{|y|^{\alpha-1}} |\nabla u(y)|\, dy \\
&=\frac{p}{N-\alpha}\int_{\R^N} \frac{|u(y)|^{p-1}}{|y|^{\alpha-1}} |\nabla u(y)|\, dy:= \frac{p}{N-\alpha} I.
\end{align}
Assume first $1 \leq \alpha <2$, take a radius $\delta >0$ and split $I=I_\delta+J_\delta$, with $I_\delta, J_\delta$ being the integrals on $B_\delta$ and $\R^N \setminus B_\delta$. Then, by H\"older's inequality, 
$$ J_\delta \leq \delta^{1-\alpha} \int_{\R^N \setminus B_\delta} |u|^{p-1} |\nabla u| \leq \delta^{1-\alpha} \left(\int_{\R^N} |u|^p \right )^{\frac 12} \left(\int_{\R^N}|\nabla u|^2 |u|^{p-2} \chi_{\{u \neq 0\}} \right )^{\frac 12}.
$$
Using $|y|^{1-\alpha} \leq \delta^{1-\frac \alpha 2} |y|^{-\frac \alpha 2}$ in $B_\delta$ we estimate, using H\"older's inequality again, 
$$
I_\delta \leq \delta^{1-\frac \alpha 2} \int_{B_\delta} \frac{|u|^{p-1}}{|y|^{\frac \alpha 2}} |\nabla u|\, dy  \leq \delta^{1-\frac \alpha 2}\left (\int_{\R^N} \frac{|u|^p}{|y|^\alpha} \right )^{\frac 12} \left ( \int_{\R^N} |\nabla u|^2|u|^{p-2} \chi_{\{ u \neq 0\}}\right)^{\frac 12}.
$$
Setting $X^2= \int_{\R^N} \frac{|u|^p}{|x|^\alpha}$, $A^2=\|u\|_p^p$, $B^2=\int_{\R^N} |\nabla u|^2 |u|^{p-2} \chi_{\{u \neq 0\}}$ we have therefore 
$$
X^2 \leq \frac{p}{N-\alpha} \left (\delta^{1-\alpha}AB+\delta^{1-\frac \alpha 2} XB \right ) \leq \frac{p}{2(N-\alpha)} \left (\delta^{1-\alpha}(\eta^{-2}A^2+\eta^2B )+\delta^{1-\frac \alpha 2} (X^2+B^2)\right )
$$
The statement then follows for $1 \leq \alpha <2$ choosing first a small $\delta$ so that $\delta^{1-\frac \alpha 2} \approx  \epsilon$ and then a small $\eta$.

The case $0 < \alpha<1$ follows now immediately from the case where $1 \leq \alpha <2$ since
\begin{align*}\int_{\R^N} \frac{|u|^p}{|x|^\alpha} \leq \int_{B} \frac{|u|^p}{|x|^{\alpha+1}}+ \int_{\R^N \setminus B}|u|^p.
\end{align*}
\end{proof}

\begin{theorem} Let $N \geq 2$, $c<0$, $1<p<\infty$. There exists $\omega_p >0$ such that $\|e^{-tS_p}\|_p \leq e^{\omega_p t}$ for $t \geq 0$.
\end{theorem}
\begin{proof}
First we consider the case where $1<p< \frac N \alpha$, so that $D(S_p)=W^{2,p}$.
We use the equality
$$-\int_{\R^N} (\Delta u)|u|^{p-2} u=(p-1)\int_{\R^N} |\nabla u|^2 |u|^{p-2} \chi_{\{u \neq 0\}}$$
which holds for every $u \in C_c^\infty$, see  \cite{MetaSpina} for the case $p<2$.  Lemma \ref{lem:quasi-accretive}  gives 
\begin{align*}\int_{\R^N} (S_p u)|u|^{p-2} u&=(p-1)\int_{\R^N} |\nabla u|^2 |u|^{p-2} \chi_{\{u \neq 0\}}+c\int_{\R^N}\frac{|u|^p}{|x|^\alpha}  \\
&\geq (p-1-c\epsilon) \int_{\R^N} |\nabla u|^2 |u|^{p-2} \chi_{\{u \neq 0\}}+cC_\epsilon \int_{\R^N} |u|^p \geq cC_\epsilon \|u\|_p^p
\end{align*}
by choosing $ c\epsilon \leq p-1$. This proves the quasi-accretivity on a core  which is equivalent to the estimate $\|e^{-tS_p}\|_p \leq e^{\omega_p t}$ with $\omega_p=-cC_\epsilon$.

Since $e^{-tS_p}$ is the adjoint of $e^{-tS_{p'}}$, the same estimate holds for every $p>\frac{N}{N-\alpha}$  with $\omega_p=\omega_{p'}$. 

Finally, the Riesz--Thorin interpolation theorem completes the proof for those $p$ (if any) between $\frac N \alpha$ and $\frac{N}{N-\alpha}$.

\end{proof}
\section{Further results and comments}
The next results is quite clear since the operator is ``radial''.

\begin{proposition} \label{radial}
For every $J \subset \N_0$, $e^{-tS_p}L^p_J \subset L^p_J$.
\end{proposition}
\begin{proof} From Lemma \ref{heatLpn} we know that $e^{t\Delta}L^p_J \subset L^p_J$ and then the domain of $\Delta$ restricted to $L^p_J$ is $W^{2,p}_J$. Assume first that $1<p< \frac N \alpha$ so that  the multiplication by $|x|^{-\alpha}$ is a small perturbation of $\Delta$ and $D(-S_p)=W^{2,p}$, see Proposition \ref{generation1}. Since the same holds in $L^p_J$, it follows that $(\lambda+S_p)W^{2,p}_J = L^p_J$ for large $\lambda$. Then the resolvent $(\lambda+S_p)^{-1}$ preserves $L^p_J$, hence the semigroup.

If $p \geq \frac N \alpha$, we choose $q< \frac N \alpha$ and use consistency. $e^{-tS_p} (L^p_J \cap L^q_J)=e^{-t S_q}(L^p_J \cap L^q_J) \subset L^q_J $. Since also $e^{-tS_p} L^p_J \subset L^p$ we have  $e^{-tS_p} (L^p_J \cap L^q_J) \subset L^p\cap L^q_J=L^p_J \cap L^q_J$ and we conclude by density.
\end{proof}

The above proposition indicates an alternative way for proving generation and domain characterization. One can prove first the result for $L^p_{\geq 2}$ using Corollary \ref{Rellichalways} to show that $-S_p$ is a small perturbation of $\Delta$, as in the proof of Proposition \ref{generation1}, and then ODE techniques for $L^p_0$ and $ L^p_1$.

\medskip

Operators with many singularities $-\Delta+\sum_k \frac{c_k}{|x-x_k|^\alpha}$ can be treated similarly as in this paper if the number of singularities is finite. In the case of infinitely many singularities one needs probably $\sup_k |c_k| <\infty$ and $\inf |x_i-x_j|>0$ to cut and paste safely the functions  $\phi(\cdot-x_i), \phi_j(\cdot-x_i)$  of Section 6. 

Finally, let us mention that spectral properties of $S_2$ are well understood, see \cite[Theorems XIII.6, XIII.82]{ReedSimon}. In particular, $S_2$ has infinitely many negative eigenvalues when $c<0$. Note also that, since the semigroup satisfies upper Gaussian estimates, the spectrum is independent of $p$. The Coulomb case corresponding to $\alpha=1, N=3$ can be found in \cite[V.12.4]{CourantHilbert}.

The following computation easily shows the existence of infinitely many negative eigenvalues when $c<0$.
Let $P_n$ be a spherical harmonic of order $n\geq 2$ and  $h_n(x)=|x|^{n}P_n(\frac{x}{|x|})$. For $\gamma>0$ set
\begin{align*}
w_n=h_n(x)e^{-\gamma |x|^{2-\alpha}}.
\end{align*}
Then we have $w_n\in W_0^{2,2}(\R^N)\subset D(S_2)$ 
for all dimension $N$ and
\[
\left(-\Delta+\frac{c}{|x|^\alpha}\right)w_n
=\left(\frac{c+\gamma(2-\alpha)(N-\alpha+2n)}{|x|^\alpha}
-(2-\alpha)^2\gamma^2|x|^{2\alpha-2}\right)w_n.
\] 
Choosing $\gamma_n=-\frac{c}{(2-\alpha)(N-\alpha+2n)}$, we have 
\[
\left(-\Delta+\frac{c}{|x|^\alpha}\right)w_n
=-\frac{c^2}{(N-\alpha+2n)^2}|x|^{2\alpha-2}w_n
\]
and therefore $(S_2w_n,w_n)<0$. Since also $(w_n, w_k)=(S_2 w_n, w_k)=0$ for $n \neq k$ we have constructed an infinite dimensional  subspace of $D(S_2)$ where the associated quadratic form  is negative. Since  $\sigma_{\rm ess}(S_2)=[0,\infty)$ (note that $|x|^{-\alpha} \in L^p(B)$ for some $p> \frac N 2$ and tends to 0 at $\infty$)
we deduce that 
$S_2$ has infinitely negative eigenvalues, by minimax theory.

\end{document}